\documentclass{amsart}
\usepackage{amssymb,amsmath,amsthm}
\usepackage[foot]{amsaddr}
\title{W-types in Homotopy Type Theory}
\author{Benno van den Berg$^1$}
\address{${}^1$ ILLC, Universiteit van Amsterdam, P.O. Box 94242, 1090 GE Amsterdam, the Netherlands. E-mail: bennovdberg@gmail.com.}
\author{Ieke Moerdijk$^2$}
\address{${}^2$ Radboud Universiteit Nijmegen, Institute for Mathematics, Astrophysics, and Particle Physics, Heyendaalseweg 135, 6525 AJ Nijmegen, the Netherlands. E-mail: i.moerdijk@math.ru.nl.}
\date{23 November, 2015}
\usepackage{ast2}
\usepackage[margin=1.7in]{geometry}
\newcommand{\op}{\ensuremath{^{\mathrm{op}}}}
\newcommand{\cat}[1]{\ensuremath{\mathbf{#1}}}
\begin{document}

\maketitle

\begin{abstract}
We will give a detailed account of why the simplicial sets model of the univalence axiom due to Voevodsky also models W-types. In addition, we will discuss W-types in categories of simplicial presheaves and an application to models of set theory.
\end{abstract}

\section{Introduction}

This paper is concerned with the interpretation of W-types in homotopy type theory. W-types are among the main type constructors in Martin-L\"of type theory, and include the type of natural numbers and many other inductive types \cite{martinlof84}. Moreover, they are an essential ingredient of Aczel's construction of a model of constructive set theory \cite{aczel78}.

Recently, Voevodsky has shown that the category of simplicial sets provides a model of type theory \cite{voevodsky11, kapulkinetal12}. In this model, types are interpreted as Kan complexes and type dependencies are interpreted as Kan fibrations. One of the main new features of this model is that it validates the univalence axiom, which gives a precise formulation of the intuitive idea that a proof of an isomorphism between types amounts to the same thing as the proof of an equality between names of these types. In this paper, we will show how W-types can be interpreted in Voevodsky's model.

In what follows we will presuppose familiarity with the simplicial sets model (for a very readable account, see \cite{kapulkinetal12}) and the classical Quillen model structure on simplicial sets (for which, see \cite{quillen67, goerssjardine99}). But we will review the categorical notion of a W-type and, in particular, its description in categories of presheaves in Section 2. In Section 3 we will show that W-types of Kan fibrations between Kan complexes are again Kan complexes.  Besides W-types, we will also discuss in Section 3 other inductive types (such as general tree types), as well as coinductive types. In Section 4 we show that the simplicial model also supports a form of quotient types and discuss the connection to Aczel's model of constructive set theory in type theory \cite{aczel78}. Finally, Section 5 will contain some remarks about how to extend these results to other model categories, in particular to certain categories of simplicial presheaves.

The main results of this paper were briefly announced at the MAP conference in Leiden (November 2011). Later, we learned that the fact that W-types are Kan (\reftheo{Wtypesinunivmodelwithdegree} below) was probably known to Voevodsky; cf. the closing sentence in \cite{voevodsky11}. Both authors wish to thank the Netherlands Organisation for Scientific Research (NWO) for financial support and the first author wishes to thank the Institute for Advanced Study for giving him the opportunity to finish this paper under such excellent working conditions. Finally, we are grateful to the referees for a careful reading of the manuscript.

\section{W-types}

We start by recalling the categorical definition of a W-type from \cite{moerdijkpalmgren00} (but see also \cite{pareschumacher78, blass83,abbottaltenkirchghani05}).

\begin{defi}{algebracoalgebrafunctor}
Let \ct{E} be a category and $F: \ct{E} \to \ct{E}$ be an endofunctor. Then an \emph{algebra} for the endofunctor $F$ consists of an object $X$ together with a map $\alpha: FX \to X$. A morphism between such algebras $(X, \alpha)$ and $(Y, \beta)$ is an arrow $f: X \to Y$ such that $f \circ \alpha = \beta \circ Ff: FX \to Y$. If it exists, the initial object in this category of $F$-algebras is the \emph{initial algebra} for the endofunctor $F$.

Dually, a \emph{coalgebra} for an endofunctor $F$ consists of an object $X$ together with a map $\alpha: X \to FX$ and a morphism of such coalgebras $(X, \alpha)$ and $(Y, \beta)$ is a map $\beta \circ f = Ff \circ \alpha: X \to FY$. And, if it exists, the \emph{final coalgebra} for the endofunctor $F$ is the terminal object in the category of $F$-coalgebras.
\end{defi}

\begin{defi}{polynomial}
Let \ct{E} be a locally cartesian closed category, and
let $f: B \to A$ be any map in \ct{E}. The \emph{polynomial functor} $P_f$ associated to $f$ is the composite
\diag{ P_f: \ct{E} \ar[r]^{ - \times B} & \ct{E}/B \ar[r]^{\Pi_f} &\ct{E}/A \ar[r]^{\Sigma_A} & \ct{E}, }
where $\Pi_f$ is the right adjoint to pulling back along $f$ and $\Sigma_A$ is the left adjoint to taking the product with $A$. If exists, the initial algebra for this endofunctor is called the \emph{W-type} associated to $f$ and denoted $W(f)$.
\end{defi}

\subsection{W-types in sets.} The category of sets and functions has all W-types. To see this, let us fix a function $f: B \to A$ and  rewrite the polynomial functor in set-theoretic notation:
\[ P_f(X) = \sum_{a \in A} X^{B_a}, \]
where $B_a = f^{-1}(a)$ is the fibre of $f$ above $a \in A$. Then the W-type consists of labelled, well-founded trees, where we imagine that the edges in the tree are directed, pointing towards the root of the tree. The idea behind the labelling is that the nodes of the tree are labelled with elements $a \in A$, while its edges are labelled with elements $b \in B$; and the labelling should be such that, if there is a node labelled with $a \in A$, then there is for every $b \in B_a$ \emph{exactly one} edge pointing towards it that has that label. The following picture hopefully conveys the idea:
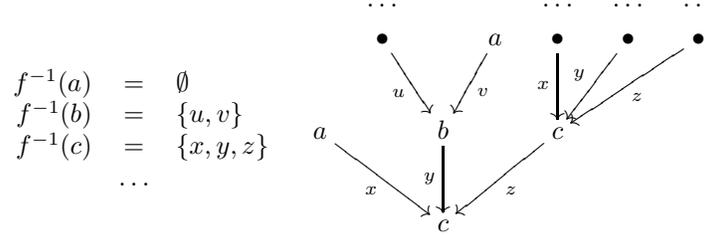
\begin{displaymath}
\xymatrix@C=.75pc@R=.5pc{ & & \ldots & & & \ldots & \ldots & \ldots \\
           & & {\bullet} \ar[dr]_u & & a \ar[dl]^v & {\bullet} \ar[d]_x & {\bullet} \ar[dl]_y & {\bullet} \ar[dll]^z \\
*{\begin{array}{rcl}
f^{-1}(a) & = & \emptyset \\
f^{-1}(b) & = & \{ u, v \} \\
f^{-1}(c) & = & \{ x, y, z \} \\
& \ldots &
\end{array}}       & a \ar[drr]_x & & b \ar[d]_y & & c \ar[dll]^z & & \\
       & & & c & & & & }
\end{displaymath}
It may not be immediately obvious that the collection of such trees is a set: but this follows from the fact that every node in the tree is uniquely determined by the finite sequence of elements in $B$ that label the edges in the path from the root to that node.

The collection $W(f)$ of such trees carries the structure of a $P_f$-algebra
\[ {\rm sup}: P_f(W(f)) \to W(f), \]
turning it into the W-type associated to $f$, as follows. If we are given an element $a \in A$ and a function $t: B_a \to W(f)$, then we can create a new tree, by taking a node, the root of the new tree, and labelling it with $a$; then, for every $b \in B_a$ we create an edge pointing towards this root, label it with $b$ and stick onto this edge the tree $t(b)$. This new tree we will denote by ${\rm sup}_a(t)$. In fact, we will think of the trees in the W-type as the result of repeatedly applying this sup-operation, possibly a transfinite number of times.

To make this idea more precise we define by transfinite recursion the notion of \emph{rank} of an element $w \in W(f)$, which is a certain ordinal. In fact, we have a map ${\rm rk}: W(f) \to Ord$ by putting
\[ {\rm rk}({\rm sup}_a (t)) = {\rm sup} \big\{ {\rm rk}(tb) + 1 \, : \, b \in B_a \big\}. \]
In addition, put
\[ W(f)_{\lt \alpha} = \{ w \in W(f) \, : \, {\rm rk}(w) \lt \alpha \}. \]
Note that $W(f)_{\lt 0} = \emptyset$ and $W(f)_{\lt \alpha + 1} \cong P_f(W(f)_{\lt \alpha})$. In addition, there exist mediating maps $W(f)_{\lt \alpha} \to W(f)_{\lt \alpha + 1}$, making $W(f)_{\lt \lambda}$ the colimit of the $W(f)_{\lt \alpha}$ for $\alpha \lt \lambda$, if $\lambda$ is a limit ordinal. This transfinite chain of sets converges to $W(f)$, for if $\kappa$ is a regular cardinal strictly bigger than all $B_a$ (for example, $({\rm sup}\{ |B_a| \, : \, a \in A \})^+$), then one proves by transfinite induction on $w \in W(f)$ that ${\rm rk}(w) \lt \kappa$; hence $W(f) = W(f)_{\lt \kappa}$. This description again makes it clear that $W(f)$ is a set, rather than a proper class.

\subsection{W-types in presheaves} Categories of presheaves also have all W-types. We will now give a concrete description, following \cite{moerdijkpalmgren00}.

Fix a category $\mathbb{C}$ and a map $f: B \to A$ between presheaves over $\mathbb{C}$. We will write
\begin{eqnarray*}
\hat{A} & = & \{ (C, a) \, : \, C \in \mathbb{C}, a \in A(C) \}
\end{eqnarray*}
and for $(C, a) \in \hat{A}$,
\begin{eqnarray*}
\hat{B}_{(C, a)} & = & \{ (\alpha: D \to C, b \in B(D)) \, : \, f_D(b) = a \cdot \alpha \}
\end{eqnarray*}
and $\hat{f}$ for the projection
\[ \sum_{(C, a) \in \hat{A}} \hat{B}_{(C,a)} \to \hat{A}. \]
As a first approximation to the W-type of $f$ in presheaves, consider the W-type $W(\hat{f})$ associated to $\hat{f}$ in the category of sets. Concretely, this means that $W(\hat{f})$ consists of well-founded trees, with nodes labelled by pairs $(C, a) \in \hat{A}$ and edges into such a node labelled with elements from $\hat{B}_{(C, a)}$, with every element from $\hat{B}_{(C, a)}$ occurring exactly once as such a label.

As it happens, we can give $W(\hat{f})$ the structure of a presheaf over $\mathbb{C}$. To do this, we will say that an element ${\rm sup}_{(C, a)} (t)$ lives in the fibre over $C$ and that for any $\alpha: D \to C$ its restriction  is given by the formula:
\[ \big( \, {\rm sup}_{(C, a)} \, (t) \, \big) \cdot \alpha = {\rm sup}_{(D, a \cdot \alpha)} \,  (t \cdot \alpha) \]
where
\[ (t \cdot \alpha)(\beta, b) = t(\alpha\beta, b). \]

As before, we can assign a rank to the elements of $W(\hat{f})$, by transfinite recursion, as follows:
\[ {\rm rk}({\rm sup}_{(C,a)} \, (t))  =  {\rm sup} \, \{ \, {\rm rk}(t(\beta, b)) + 1  \, : \, (\beta, b) \in \hat{B}_{(C,a)} \, \}. \]
Note that if $w \in W(\hat{f})(C)$ and $\alpha: D \to C$, then ${\rm rk}(w \cdot \alpha) \leq {\rm rk}(w).$ Therefore
\[ W(\hat{f})_{\lt \alpha} = \{ w \in W(\hat{f}) \, : \, {\rm rk}(w) \lt \alpha \} \]
defines a subpresheaf of $W(\hat{f})$.

The W-type associated to $f$ is constructed by selecting those elements from $W(\hat{f})$ that are \emph{hereditarily natural}.

\begin{defi}{natandcomptrees} A tree ${\rm sup}_{(C, a)}(t)$ is \emph{composable}, if for any $(\alpha: D \to C, b) \in \hat{B}_{(C, a)}$, the tree $t(\alpha, b)$ lives in the fibre over ${\rm dom}(\alpha)$. If, in addition, the map $t$ is a natural transformation, meaning that for any $(\alpha: D \to C, b) \in \hat{B}_{(C, a)}$ and $\beta: E \to D$ we have
\[ t(\alpha\beta, b \cdot \beta) = t(\alpha, b) \cdot \beta, \]
then the tree ${\rm sup}_{(C, a)}(t)$ will be called \emph{natural}.

The collection of \emph{subtrees} of ${\rm sup}_{(C, a)}(t)$ is defined recursively as the collection consisting of ${\rm sup}_{(C, a)}(t)$ and all the subtrees of the $t(\alpha, b)$. Finally, a tree will be called \emph{hereditarily natural}, if all its subtrees are natural.
\end{defi}

Since any restriction of an hereditarily natural tree is again hereditarily natural, the hereditarily natural trees form a subpresheaf $W(f)$ of $W(\hat{f})$. This defines the W-type in presheaves associated to $f$. In addition, we will put
\[ W(f)_{\lt \alpha} = \{ w \in W(f) \, : \, {\rm rk}(w) \lt \alpha \} = W(\hat{f})_{\lt \alpha} \cap W(f) \subseteq W(\hat{f}). \]
As the intersection of two presheaves, this is again a presheaf. In fact, we again have that $W(f)_{\lt 0} = 0$, that $W(f)_{\lt \alpha + 1} = P_f(W(f)_{\lt \alpha})$, and that $W(f)_{\lt \lambda}$ is the colimit of the $W(f)_{\lt \alpha}$ where $\alpha$ is an ordinal smaller than the limit ordinal $\lambda$. In addition, this chain again converges to $W(f)$; indeed, by choosing $\kappa$ large enough (regular and greater than $|\hat{B}_{(C, a)}|$ for all $(C, a) \in \hat{A}$), we get $W(f) = W(f)_{\lt \kappa}$.

\subsection{Variations} The ideas from the previous paragraphs allow for numerous variations. For example, there are the dependent polynomial functors of Gambino and Hyland (see \cite{gambinohyland04}; this is related to the general tree types of Petersson and Synek \cite{peterssonsynek89}).

\begin{defi}{dependentpolynomial}
Suppose we are given a diagram of the form
\diag{ B \ar[d]_h \ar[r]^f & A \ar[d]^g \\
C & C }
in a locally cartesian closed category \ct{E}.
Then this diagram determines an endofunctor on $\ct{E}/C$,  the \emph{dependent polynomial functor}
\diag{ D_f: \ct{E}/C \ar[r]^(.6){h^*} & \ct{E}/B \ar[r]^{\Pi_f} &\ct{E}/A \ar[r]^{\Sigma_g} & \ct{E}/C. }
\end{defi}

Also functors of the form $D_f$ have initial algebras in the category of sets. To see this, let us first rewrite $D_f$ in set-theoretic notation:
\[ D_f(X)_c = \sum_{a \in A_c} \prod_{b \in B_a} X_{h(b)}. \]
Then its initial algebra is obtained from the W-type of $f$ by selecting from $W(f)$ those trees which satisfy the following additional compatibility condition: if an edge is labelled with some $b \in B$ and the source of this edge is a node labelled with $a \in A$, then we should have $g(a) = h(b)$. As a subset of the W-type, elements in this initial algebra again have a rank; and the initial algebra can be seen as the result of repeatedly applying the $D_f$ operation, starting from the empty set and possibly applying $D_f$ a transfinite number of times. Similar remarks hold for categories of presheaves: initial algebras for dependent polynomial also exist; indeed, they are suitable subobjects of the W-type associated to $f$ and as such also inherit a notion of rank.

Instead of looking at initial algebras, we could also look at final coalgebras.

\begin{defi}{Mtype}
Let \ct{E} be a locally cartesian closed category, and
let $f: B \to A$ be any map in \ct{E}. If it exists, the final coalgebra of the polynomial functor associated to $f$ is called the \emph{M-type} associated to $f$ and denoted $M(f)$.
\end{defi}
M-types also exist both in sets and in presheaves (see \cite{bergdemarchi07a}). The idea here is that we look at trees with the kind of labelling described at the beginning of the section: nodes labelled with elements $a \in A$, edges labelled with elements $b \in B$, in such a way that $B_a$ enumerates the edges into a node labelled with $a \in A$. But the difference is that the M-type consists of \emph{all} such trees, including those that are not well-founded.

Dually, these M-types can be obtained as a limit of a chain:
\diag{ \ldots \ar[r] & P_f(P_f(P_f(1))) \ar[r] & P_f(P_f(1)) \ar[r] & P_f(1) \ar[r] & 1.}
One big difference is that this chain stabilises already at the ordinal $\omega$; in other words, $M(f)$ is the limit of the $P_f^n(1)$ with $n \in \NN$. To see this, write $\tau$ for the coalgebra map $\tau: M(f) \to P_f(M(f))$ and define for every $n \in \NN$ a truncation function $tr_n: M(f) \to P_f^n(1)$, by letting $tr_0$ be the unique map $M(f) \to 1$, and $tr_{n+1}$ be the composite
\diag{ tr_{n+1}: M(f) \ar[r]^(.55){\tau} & P_f(M(f)) \ar[rr]^{P_f(tr_n)} & & P_f^{n+1}(1). }
What the $n$th truncation does is cutting off the tree at level $n$ and replacing the subtrees that have disappeared with the unique element of $1$. To see that the $tr_n: M(f) \to  P_f^n(1)$ form a colimiting cone, the key observation is that every tree is completely determined by its $n$th truncations. And all of this is equally true in categories of presheaves.

\section{Simplicial sets}

In this section we will study W-types in $\SSets$, the category of simplicial sets, in particular in connection with the univalent model of type theory. This univalent model uses the Quillen model structure on simplicial sets \cite{quillen67, quillen69}; of course, it carries several such, but the relevant one here is the classical model structure due to Quillen, in which:
\begin{itemize}
\item weak equivalences are those maps whose geometric realizations are homotopy equivalences.
\item fibrations are those maps that have the right lifting property with respect to horn inclusions (\emph{aka} Kan fibrations).
\item cofibrations are the monomorphisms.
\end{itemize}

As simplicial sets form a presheaf category, the previous section gives us a clear picture of how the W-types look there. The main result of this section will be that if $f: B \to A$ is a Kan fibration, then so is the canonical map $W(f) \to A$. But to prove this we need to know a few more things beyond the fact that the three classes of maps defined above give simplicial sets the structure of a Quillen model category.

\subsection{Properties of the classical model structure on simplicial sets}

For the proof we need the following properties of the standard model structure on simplicial sets:

\begin{prop}{rightproper}
Trivial cofibrations are stable under pullback along Kan fibrations.
\end{prop}
\begin{proof}
Since the cofibrations are the monomorphisms and hence stable under pullback along any map, it suffices to show that the weak equivalences are stable under pullback along fibrations; \emph{i.e.}, that the model structure is right proper. This is well-known: in fact, it follows from the fact that geometric realization preserves pullbacks, maps Kan fibrations to Serre fibrations \cite{quillen68}, and homotopy equivalences are stable under pullback along Serre fibrations.
\end{proof}

\begin{coro}{univalentPi}
If $f: B \to A$ is a Kan fibration, then $\prod_f: \SSets/B \to \SSets/A$ preserves Kan fibrations.
\end{coro}
\begin{proof}
A straightforward diagram chase.
\end{proof}

\begin{prop}{filteredcolimit}
If $X$ is the filtered colimit of $(X_i \, : \, i \in I)$ and each $X_i \to A$ is a Kan fibration, then so is the induced map $X \to A$.
\end{prop}
\begin{proof}
This is immediate from the fact that Kan fibrations are maps which have the right lifting property with respect to horn inclusions and horns are finite colimits of representables.
\end{proof}

\subsection{W-types in simplicial sets} The main result of this section is:

\begin{theo}{Wtypesinunivmodelwithdegree}
If $f: B \to A$ is a Kan fibration between Kan complexes, then for any ordinal $\alpha$ the map $W(f)_{\lt \alpha} \to A$ is also a Kan fibration; in particular, $W(f) \to A$ is.
\end{theo}
\begin{proof}
First of all, more generally, we claim that if $Z$ is a Kan complex then $P_f(Z) \to A$ is a Kan fibration. Indeed, let $X \to Y$ be a trivial cofibration and suppose we have a commuting square
\diag{ X \ar[d] \ar[r]^(.4)K & P_f(Z) \ar[d] \\
Y \ar[r] & A.}
We want to find a map $L: Y \to P_f(Z)$ which makes the two resulting triangles commute. Note that $K$ transposes to a map $k: B \times_A X \to W(f)_{\lt \alpha}$ and $B \times_A  X \to B \times_A Y$ is a trivial cofibration by \refprop{rightproper}. So if $Z$ is Kan, there is a dotted arrow $l$ making
\diag{ B \times_A X \ar[d] \ar[r]^k & Z \\
B \times_A Y \ar@{.>}[ur]_l  }
commute. Taking the transpose of $l$ gives us the desired map $L$.

To prove the theorem, we argue by induction, the case of a limit ordinal $\alpha$ (including $\alpha = 0$) being clear from \refprop{filteredcolimit}. So suppose $W(f)_{\lt \alpha} \to A$ is a Kan fibration. Since $A$ is assumed Kan, $W(f)_{\lt \alpha}$ is too. So by the claim, $W(f)_{\lt \alpha + 1} = P_f(W(f)_{\lt \alpha})\to A$ is Kan.

Since $W(f) = W(f)_{\lt \alpha}$ for sufficiently large $\alpha$, we have as a special case that $W(f) \to A$ is a Kan fibration.
\end{proof}

\subsection{Variations} An easy variation on the previous result would be, for example:

\begin{theo}{initalgdepfibrant}
If we have a diagram
\diag{ B \ar[r]^f \ar[d]_h & A \ar[d]^g \\
C & C}
of Kan fibrations in simplicial sets, then the initial $D_f$-algebra is fibrant in $\SSets/C$.
\end{theo}
\begin{proof}
The general picture is really this: suppose $\Phi$ is an endofunctor on the category of simplicial sets, or any other model category in which fibrant objects are closed under directed colimits. If this endofunctor sends fibrant objects to fibrant objects and has an initial algebra which can be built as the colimit of a sufficiently long chain of $\Phi^\alpha(0)$, then this initial algebra has to be fibrant as well. By considering $D_f$ on $\SSets/C$ we obtain the desired result.
\end{proof}

Dually we have:

\begin{theo}{Mtypesinunivmodel}
If $f: B \to A$ is a Kan fibration between fibrant objects, then $M(f)$ is fibrant as well.
\end{theo}
\begin{proof}
Here the general picture is: suppose $\Phi$ is an endofunctor on the category of simplicial sets, or any other model category, which preserves fibrations and for which $\Phi(1)$ is fibrant. If $\Phi$ has a final coalgebra and it can be obtained as a limit of a sufficiently long chain of $\Phi^\alpha(1)$, then this final coalgebra is fibrant. The desired result follows by specialising to the case $\Phi = P_f$.
\end{proof}

\section{Quotients}

In this section we discuss quotients of equivalence relations on simplicial sets. We will show that the simplicial model of univalent foundations supports a form of quotient types, sufficient for constructing a model of Aczel's constructive set theory.

\subsection{Quotient types in the univalent model.} We first observe:
\begin{prop}{quotients}
If in a commutative triangle
\diag{ Y \ar[rr]^p \ar[dr]_g & & X \ar[dl]^f \\
& A }
with $p$ epic, both $p$ and $g$ are Kan fibrations, then so is $f$.
\end{prop}
\begin{proof}
Consider a commuting square
\diaglab{desired}{ \Lambda^k[n] \ar[d]_i \ar[r]^\alpha & X \ar[d]^f \\
\Delta[n] \ar[r]_\beta & A }
with a horn inclusion $i$ on the left. As $1 = \Delta[0]$ is representable and $p$ is epic, there is a map $\gamma$ making the square
\diag{ \Delta[0] \ar[d]_k \ar[r]^\gamma & Y \ar[d]^p \\
\Lambda^k[n] \ar[r]_\alpha & X }
commute, where $k: \Delta[0] \to \Lambda^k[n]$ picks the $k$th vertex. Note that $k$ is a strong deformation retract and hence a trivial cofibration; in addition, the map $p$ is fibration by assumption, so this square has a diagonal filler $\delta$. But then
\diag{ \Lambda^k[n] \ar[d]_i \ar[r]^\delta & Y \ar[d]^g \\
\Delta[n] \ar[r]_\beta & A }
commutes, so has a diagonal filler $\epsilon$. Now $p\epsilon$ is a diagonal filler for \refdiag{desired}, as:
\[ p \epsilon i = p \delta = \alpha \quad \mbox{ and } \quad f p \epsilon = g \epsilon = \beta. \]
\end{proof}

\begin{coro}{images}
If $f: Y \to X$ is a Kan fibration, then so are the maps in its factorisation as an epi $p$ followed by a mono $i$.
\end{coro}
\begin{proof}
For $p$ this is clear and for $i$ this follows from the previous proposition.
\end{proof}

\begin{prop}{eqrel}
If $R$ is an equivalence relation on $Y$ and both projections $R \to Y$ are fibrations, then $Y \to Y/R$ is a fibration as well.
\end{prop}
\begin{proof}
Consider a commuting square
\diaglab{desired2}{ \Lambda^k[n] \ar[d]_i \ar[r]^\alpha & Y \ar@{->>}[d]^q \\
\Delta[n] \ar[r]_\beta & Y/R }
with a horn inclusion $i$ on the left. As $\Delta[n]$ is representable and $q$ is epic, there is a map $\gamma: \Delta[n] \to Y$ such that $q\gamma = \beta$. We do not necessarily have $\gamma i = \alpha$, but we do have that $q \gamma i = q \alpha$ (because both are equal to $\beta i$). So we have a commuting square
\diag{ \Lambda^k[n] \ar[d]_i \ar[r]^(.55){(\alpha, \gamma i) }& R \ar[d]^{\pi_2} \\
\Delta[n] \ar[r]_\gamma & Y, }
in which there must exist a diagonal filler $\delta$. Now $\pi_1 \delta$ is a diagonal filler for \refdiag{desired2}, as:
\[ \pi_1 \delta i = \pi_1 (\alpha, \gamma i) = \alpha \quad \mbox{ and } \quad q \pi_1 \delta = q \pi_2 \delta = q \gamma = \beta. \]
\end{proof}

To state the main result of this subsection, we recall from \cite{carboniceliamagno82, carboni95} that $(s, t): R \to Y \times Y$ is a \emph{pseudo-equivalence relation}, if:
\begin{enumerate}
\item there is a map $\rho: Y \to R$ such that $(s, t) \rho$ is the diagonal map $\Delta_Y: Y \to Y \times Y$.
\item there is a map $\sigma: R \to R$ such that $s\sigma = t$ and $t\sigma = s$.
\item if $P$ is the pullback
\diag{ P \ar[r]^{p_{12}} \ar[d]_{p_{23}} & R \ar[d]^t \\
R \ar[r]_s & Y, }
then there is a map $\tau: P \to R$ such that $s p_{12} = s \tau$ and $t p_{23} = t \tau$.
\end{enumerate}

\begin{coro}{pseudoeqrel}
Suppose $R$ is a pseudo-equivalence relation on a object $Y$ and $R \to Y \times Y$ is a Kan fibration. If $Y$ is fibrant, then so is $Y/R$ and the quotient map $Y \to Y/R$ is a Kan fibration.
\end{coro}
\begin{proof}
Without loss of generality we may assume that $R \to Y \times Y$ is monic: for otherwise we may replace $R \to Y \times Y$ by its image $S \subseteq Y \times Y$. This inclusion is again a Kan fibration by \refcoro{images} and the quotients $Y/R$ and $Y/S$ are isomorphic.

So assume $R \to Y \times Y$ is monic. Then it is an equivalence relation, and since $Y$ is fibrant, the projections $Y \times Y \to Y$ are Kan fibrations, and so are the projections $R \to Y$. So $Y \to Y/R$ is a Kan fibration by the previous proposition and $Y/R$ is fibrant according to \refprop{quotients}.
\end{proof}

\subsection{Application} Voevodsky has shown that if one restricts the Kan fibrations to those that have small fibres (for example, those whose fibres have a cardinality smaller than some inaccessible cardinal $\kappa$), then there is a generic small Kan fibration $\pi: E \to U$; that is, there is a Kan fibration with small fibres $\pi$ such that any other Kan fibration with small fibres can be obtained as a pullback of $\pi$. In addition, the object $U$ can be chosen to be fibrant (see \cite{voevodsky11,kapulkinetal12}).

We can use this generic Kan fibration $\pi$ to construct a model of constructive set theory: this is sometimes called the Aczel construction. It was originally discovered by  Peter Aczel in a type-theoretic context \cite{aczel78} and it was reformulated categorically in \cite{moerdijkpalmgren02}. The idea is to take the W-type associated to $\pi$ and then quotient by bisimulation.

So take $W(\pi)$, the W-type associated to $\pi$, and define the following endofunctor $\Phi$ on $\SSets/W(\pi) \times W(\pi)$:
\[ \Phi(X)_{{\rm sup}_u(t), {\rm sup}_{u'}(t')} = \prod_{e \in E_u} \sum_{e' \in E_{u'}} X_{t(e), t'(e')} \times \prod_{e' \in E_{u'}} \sum_{e \in E_{u}} X_{t(e), t'(e')}, \]
where we have used set-theoretic notation. This defines a dependent polynomial functor on $\SSets/W(\pi) \times W(\pi)$, for which we can take its initial algebra $B \to W(\pi) \times W(\pi)$: here we should think of an element in the fibre over a pair $({\rm sup}_u(t), {\rm sup}_{u'}(t'))$ as the type of proofs of the bisimilarity of ${\rm sup}_u(t)$ and ${\rm sup}_{u'}(t')$. This map $B \to W(\pi) \times W(\pi)$ is a pseudo-equivalence relation (as one may easily verify) and a Kan fibration by \reftheo{initalgdepfibrant}. Since $W(\pi)$ is fibrant by \reftheo{Wtypesinunivmodelwithdegree}, its quotient must be fibrant as well, by \refcoro{pseudoeqrel}. This means that if we perform the Aczel construction in the univalent model of type theory, we get a fibrant model of constructive set theory.

One may also dualize and take the M-type on $\pi$ and then quotient by the largest bisimulation (as in \cite{lindstrom89} and \cite{bergdemarchi07b}). This should result in a fibrant model of constructive set theory satisfying Aczel's Anti-Foundation Axiom \cite{aczel88}.

\section{Other model categories}

As we have seen above, the Quillen model category of simplicial sets provides an interpretation of Martin-L\"of type theory including W-types. The argument relied on the fact that W-types can be obtained by repeatedly, and possibly transfinitely, applying the polynomial to the initial object, as well as the fact that it is a model category \ct{E} for which:
\begin{enumerate}
\item[(1)] Trivial cofibrations are stable under pullback along fibrations in \ct{E}.
\item[(2)] If $X$ is the filtered colimit of $\{ X_i \, : \, i \in I \}$ and each $X_i \to A$ is a fibration, then so is the induced map $X \to A$.
\end{enumerate}
We recall that property (1) is equivalent to (1$'$), and is a consequence of the combined properties (1a) and (1b), which also hold in $\SSets$:
\begin{enumerate}
\item[(1$'$)] If $f: B \to A$ is a fibration then the right adjoint $\Pi_f: \ct{E}/B \to \ct{E}/A$ to the pullback functor preserves fibrant objects.
\item[(1a)] The cofibrations in \ct{E} are exactly the monomorphisms.
\item[(1b)] \ct{E} is right proper.
\end{enumerate}

Unfortunately, when trying to extend the argument to categories of simplicial presheaves, one discovers that these two conditions (1) and (2) generally seem to have rather incompatible stability properties. For example, while property (1) evidently transfers to the injective model structure on a category $\SSets^{\ct{C}\op}$ of simplicial presheaves, property (2) rarely does. And while property (2) evidently transfers to the projective model structure on simplicial presheaves, property (1) generally does not. One of the few exceptions to this is the case where \ct{C} is a group:

\begin{exam}{grpactions} (Group actions) Let $G$ be a group, and let $\SSets_G$ be the category of simplicial sets with right $G$-action. This category carries a (cofibrantly generated) model structure, with the property that the forgetful functor
\[ U: \SSets_G \to \SSets \]
preserves and reflects weak equivalences and fibrations. Since this forgetful functor commutes with $\Pi$-functors and filtered colimits, the category $\SSets_G$ again has properties (1) and (2). One can also check property (1) directly, since the cofibrations in $\SSets_G$ are the monomorphisms $X \to Y$ with the property that $G$ acts freely on the simplices of $Y$ which are not in (the image of) $X$.
\end{exam}

\begin{exam}{reedystr} (Reedy categories) We recall that a Reedy category is a category $\mathbb{R}$ equipped with two classes of maps $\mathbb{R}^-$ and $\mathbb{R}^+$ which both contain all the identities and are closed under composition, and a degree function $d: {\rm Objects}(\mathbb{R}) \to \NN$ for which
\begin{enumerate}
\item[(i)] any non-identity morphism in $\mathbb{R}^+$ raises degree, and any non-identity morphism in $\mathbb{R}^-$ lowers degree;
\item[(ii)] every morphism in $\mathbb{R}$ factors uniquely as a morphism in $\mathbb{R}^-$ followed by one in $\mathbb{R}^+$.
\end{enumerate}
If \ct{E} is a model category and $\mathbb{R}$ is a Reedy category, the functor category $\ct{E}^{\mathbb{R}}$ carries a model structure in which the weak equivalences are defined ``pointwise''; i.e., $X \to Y$ is a weak equivalence iff $X_r \to Y_r$ is for every $r \in \mathbb{R}$. The special virtue of this ``Reedy model structure'' is that the fibrations and cofibrations can be described explicitly in terms of so-called matching and latching objects. If $X$ is an object of $\ct{E}^{\mathbb{R}}$, the $r$th matching and latching objects of $X$ are defined as
\[ M_r(X) = \varprojlim_{r \xrightarrow{-} s} X_s \quad \mbox{ and } \quad L_r(X) = \varinjlim_{s \xrightarrow{+} r} X_s, \]
where the limit and colimit are taken over the non-identity maps in $\mathbb{R}^-$ and $\mathbb{R}^+$ respectively. A map $Y \to X$ is a \emph{fibration} in $\ct{E}^{\mathbb{R}}$ if, for any object $r \in \mathbb{R}$, the map
\[ Y_r \to X_r \times_{M_r(X)} M_r(Y) \]
is a fibration in \ct{E}. And, dually, a map $A \to B$ is a cofibration in $\ct{E}^{\mathbb{R}}$ if, for any object $r \in \mathbb{R}$, the map
\[ A_r \cup_{L_r(A)} L_r(B) \to B_r \]
is a cofibration in \ct{E}.

Typical examples are the simplex category $\Delta$ where $d([n]) = n$, while $\mathbb{R}^-$ consists of the surjections and $\mathbb{R}^+$ consists of the injections, or the category $\NN$ itself viewed as a poset (with $\NN = \NN^+$); the opposite categories $\Delta\op$ and $\NN\op$ are also Reedy categories, with $\mathbb{R}^+$ and $\mathbb{R}^-$ simply interchanged. In these examples and many others, the limits and colimits involved in the matching and latching objects are (essentially) finite. Let us say that a Reedy category $\mathbb{R}$ is \emph{locally finite} if each comma category $r/ \mathbb{R}^-$ contains a finite cofinal subcategory, so that the matching objects are defined by finite limits. Then clearly, if $\mathbb{R}$ is locally finite and $\ct{E}$ is a model category satisfying condition (2), then so does $\ct{E}^\mathbb{R}$.

Condition (1) seems to be less well-behaved with respect to arbitrary Reedy model structures. However, in many important examples the Reedy cofibrations in $\ct{E}^\mathbb{R}$ turn out to be the pointwise cofibrations. This is trivially the case if the category is ``inverse'', that is, if $\mathbb{R} = \mathbb{R}^-$ (and $\mathbb{R}^+$ contains identities only) as in $\NN\op$, but it also holds for ``elegant'' Reedy categories (see \cite{bergnerrezk12}) such as $\mathbb{R} = \Delta\op$. For now, let us state the following:
\begin{prop}{Reedycool}
Let \ct{E} be a model category satisfying conditions (1a, b) and (2). If $\mathbb{R}$ is a locally finite Reedy category for which the cofibrations in $\ct{E}^\mathbb{R}$ are pointwise (for example, if $\mathbb{R} = \Delta\op$ or $\mathbb{R} = \mathbb{N}\op$), then $\ct{E}^\mathbb{R}$ again satisfies these conditions.
\end{prop}
\end{exam}
\begin{exam}{genreedy} (Generalised Reedy categories)
Although extremely useful in homotopy theory, the notion of Reedy category has various defects: it is not invariant under equivalence of categories, and excludes categories with non-trivial automorphisms. There is, however, a notion of ``generalised Reedy category'' which allows for the same construction of a model structure on $\ct{E}^\mathbb{R}$ from one on \ct{E}, and is more flexible. In particular, it includes important examples like the category $\cat{Fin}$ of finite sets, the category $\cat{Fin}_*$ of finite pointed sets (or equivalently, finite sets and partial maps) and the category $\Omega$ of trees. We refer to \cite{bergermoerdijk11} for details.

As before, (2) will not preserved in general, but it will be inherited whenever $\mathbb{R}$ is locally finite. Also (1) will hold in several important examples where they are satisfied. In fact, let $\mathbb{R}$ be a ``dualisable'' generalised Reedy category (cf.~\cite{bergermoerdijk11}) such as $\Delta, \Omega, \cat{Fin}_*$ and $\cat{Fin}$. It is perhaps useful to be more explicit about the cofibrations in the Reedy model structure on $\ct{E}^{\mathbb{R}\op}$. (We have passed to contravariant functors here because it fits the examples better.) First of all, recall from \cite{bergermoerdijk11} that in the case of a generalised Reedy category, the automorphisms of $\mathbb{R}$ and the model structure of \refexam{grpactions} enter into the description of the cofibrations. In particular, a map $X \to Y$ is a Reedy cofibration in $\ct{E}^{\mathbb{R}\op}$ iff for each object $r \in \mathbb{R}$, the map
\[ L_r(Y) \cup_{L_r(X)} X_r \to Y_r \]
is a cofibration in $\ct{E}^{{\rm Aut}(r)\op}$. Because we have passed to the dual $\mathbb{R}\op$, the latching object is now described as
\[ L_r(X) = \varinjlim_{r \xrightarrow{-} s} X_s, \]
the colimit ranging over all non-isomorphic maps $r \to s$ in $\mathbb{R}^-$; the ``surjections'' in the examples.  If $\ct{E} = \SSets$, then the cofibrations in $\ct{E}^{{\rm Aut}(r)\op}$ are characterised as the monos with ``free action on the complement'', as in \refexam{grpactions}; since this property is preserved by pullbacks, cofibrations will be stable under pullback. As these examples are also right proper, (1) will hold in these examples. In particular, this applies to the category $\SSets^{\Omega\op}$ of dendroidal spaces, $\SSets^{\Gamma\op}$ of $\Gamma$-spaces, and $\SSets^{\cat{Fin}\op}$ of symmetric simplicial sets. Hence these all satisfy properties (1) and (2) stated at the beginning of the section.
\end{exam}

\bibliographystyle{plain} \bibliography{WinHoTT}

\end{document}